\documentclass[11pt, reqno]{amsart}
\usepackage{hyperref}
\usepackage{amssymb}
\usepackage{calrsfs}
\newcommand\R{\mathbf R}
\newcommand\C{\mathbf C}

\newcommand\LA{{\mathsf A}}
\newcommand\LB{{\mathsf B}}
\newcommand\LC{{\mathsf C}}
\newcommand\LD{{\mathsf D}}
\newcommand\LE{{\mathsf E}}
\newcommand\LF{{\mathsf F}}
\newcommand\LG{{\mathsf G}}

\newcommand\La{\mathfrak a}

\newcommand\SO{\operatorname{\sf SO}}
\newcommand\Sp{\operatorname{\sf Sp}}
\newcommand\SU{\operatorname{\sf SU}}

\newcommand\Spin{\operatorname{\sf Spin}}
\newcommand\U{\operatorname{\sf U}}

\newcommand\so{{\mathfrak{so}}}
\newcommand\su{{\mathfrak{su}}}

\renewcommand\sp{{\mathfrak{sp}}}
\renewcommand\u{{\mathfrak{u}}}
\renewcommand\k{{\mathfrak{k}}}
\renewcommand\t{{\mathfrak{t}}}

\newcommand\fsp{{\mathfrak{sp}}}
\newcommand\z{{\mathfrak{z}}}

\newcommand\op{{\oplus}}
\newcommand\eS{\operatorname{\sf S}}

\newcommand\RR{{\mathcal R}}
\newcommand\PP{{\mathcal P}}

\newcommand\g{{\mathfrak{g}}}
\newcommand\h{{\mathfrak{h}}}
\renewcommand\l{{\mathfrak{l}}}
\newcommand\m{{\mathfrak{m}}}

\newcommand\p{{\mathfrak{p}}}
\newcommand\s{{\mathfrak{s}}}

\newcommand\im{{\operatorname{im}}}

\newcommand\id{\operatorname{id}}

\newcommand\tr{\operatorname{tr}}
\newcommand\Aut{\operatorname{Aut}}
\newcommand\diag{\operatorname{diag}}
\newcommand\Ad{\operatorname{Ad}}
\newcommand\ad{\operatorname{ad}}

\newcommand\Lie{\operatorname{Lie}}
\newcommand\rk{\operatorname{rk}}

\theoremstyle{plain}
\newtheorem{theorem}{Theorem}

\newtheorem{lemma}[theorem]{Lemma}
\newtheorem{proposition}[theorem]{Proposition}
\newtheorem{corollary}[theorem]{Corollary}

\theoremstyle{definition}

\theoremstyle{remark}
\newtheorem{remark}[theorem]{Remark}

\numberwithin{equation}{section}
\numberwithin{theorem}{section}

\begin{document}
\title[Nonnegatively curved homogeneous metrics]
{Nonnegatively curved homogeneous metrics obtained by scaling fibers of
submersions}

\subjclass[2000]{Primary: 53C30; Secondary: 53C21; 57S15}

\author[M. M. Kerr]{Megan M. Kerr}
\address{Department of Mathematics, Wellesley College, 106 Central St.,
Wellesley, MA 02481} \email{mkerr@wellesley.edu}

\author[A. Kollross]{Andreas Kollross}
\address{Fachbereich Mathematik, Universit\"{a}t Stuttgart,
Pfaffenwaldring 57, D-70569 Stuttgart}
\email{kollross@mathematik.uni-stuttgart.de}
\date{\today}

\begin{abstract}
We consider invariant Riemannian metrics on compact homogeneous spaces $G/H$
where an intermediate subgroup $K$ between $G$ and $H$ exists, so that
the homogeneous space $G/H$ is the total space of a Riemannian submersion. We
study the question as to whether enlarging the fibers of the submersion by a
constant scaling factor retains the nonnegative curvature in the case that the
deformation starts at a normal homogeneous metric. We classify triples of
groups $(H,K,G)$ where nonnegative curvature is maintained for small
deformations, using a criterion proved by Schwachh\"{o}fer and Tapp. We obtain a
complete classification in case the subgroup $H$ has full rank and an almost
complete classification in the case of regular subgroups.
\end{abstract}
\maketitle


\section{Introduction}
The study of manifolds with nonnegative or positive sectional curvature is one
of the classical fields of Riemannian geometry. Examples of manifolds which
admit metrics of strictly positive curvature are scarce, but if one relaxes the
curvature condition to include manifolds of nonnegative curvature, the
situation is different.
For instance, any compact
homogeneous space admits a metric of nonnegative sectional curvature.

The class of manifolds supporting Riemannian metrics of nonnegative curvature
is larger; nevertheless, only a few methods are known and it is also of
interest to explore the family of all Riemannian metrics with nonnegative
curvature on a given manifold. Some results in this connection were recently
obtained by Schwachh\"{o}fer and Tapp~\cite{ST} in the setting of compact
homogeneous spaces. They investigate certain deformations of a normal
homogeneous metric on a compact homogeneous space~$G/H$ within the class of
$G$-invariant metrics.

They prove the following structural result. The family of invariant metrics is
star-shaped with respect to any normal homogeneous metric if the symmetric
matrices corresponding to invariant metrics are parametrized by their inverses.
Thus the problem of determining all invariant metrics with nonnegative
curvature reduces to determining how long nonnegative curvature is maintained
when deforming along a linear path, starting at a normal homogeneous metric.

Schwachh\"{o}fer and Tapp investigate this problem for the simplest nontrivial
case, namely they assume there is an intermediate subgroup~$K$ between $H$ and
$G$ and study metrics obtained through shrinking or enlarging the fibers of the
Riemannian submersion $G/H \to G/K$ by a constant scaling factor. As they point
out, a metric obtained by shrinking the fibers can be interpreted as a
submersion metric obtained via a Cheeger deformation~\cite{Ch} and hence
shrinking the fibers always preserves nonnegative curvature. On the other hand,
if one enlarges the fibers by a constant scaling factor, whether nonnegative
curvature is maintained for small deformations depends on the triple $(H,K,G)$.
In~\cite{ST} they find a criterion on the Lie algebra level, see
Theorem~\ref{ThmST} below.
The preservation of nonnegative curvature under scaling up first appears  in
Grove and Ziller's paper on Milnor spheres \cite{GZ}.

This condition holds in particular if $(K,H)$ is a symmetric pair, an
observation which yields a new class of examples for nonnegatively curved
metrics. To study the Lie-theoretic condition found by Schwachh\"{o}fer and Tapp is
interesting in its own right and determining which triples $(H,K,G)$ of compact
Lie groups satisfy the criterion turns out to be an intriguing problem. In
\cite{kk} we classify all such triples in the special case where $G$ is simple
of dimension up to~15. In the present article, we use root space decompositions
to study the problem for three classes of examples.

The first class consists of triples $(H,K,G)$ where $G/K$ is a symmetric space
with $\rk(G/K) = \rk(G)$ and $H$ arises as the intersection of~$K$ with a
subgroup of maximal rank in~$G$. The second class consists of all triples
$(H,K,G)$ for which $\rk(H)=\rk(K)=\rk(G)$. For those two classes of spaces we
prove that enlarging the fibers of the submersion $G/H \to G/K$ maintains
nonnegative curvature if and only if $(K,H)$ is a symmetric pair. In the last
part of this article we consider more generally the criterion for triples
$(H,K,G)$ where $H, K$ are regular subgroups of~$G$. Here there are also
examples satisfying the criterion such that $(K,H)$ is not a symmetric pair.
Some of these examples were known before by work of \cite{ST}, but we also
present a new class of examples, see Section~\ref{NewExamples}. We obtain an
almost complete classification in this case.

It is a pleasure to thank the referee for helpful suggestions.

\section{Preliminaries}
\label{Prelim}

Let $H \subsetneq K \subsetneq G$ be compact Lie groups and let $\h, \k, \g$,
be their respective Lie algebras. Let $g_0$ be a biinvariant inner product
on~$\g$. Let $\s$ be the complement of~$\k$ in~$\g$ and let $\m$ be the
complement of~$\h$ in~$\k$ such that $\g = \h \oplus \m \oplus \s$ is an
orthogonal decomposition with respect to $g_0$. Set $\p = \m \oplus \s$. Then
$\p$ can be identified with the tangent space of the homogeneous manifold $G/H$
at the point $1H$. The restriction of $g_0$ to~$\p$ defines an
$\Ad_H$-invariant inner product and thus a $G$-invariant Riemannian metric
on~$G/H$, which we will also denote by $g_0$.

For any element $X$ in $\p$, we write $X = X^{\m} + X^{\s}$, where $X^{\m} \in
\m$ denotes the vertical component and $X^{\s} \in \s$ denotes the horizontal
component of~$X$. We study $G$-invariant metrics on $G/H$ which are obtained by
deforming the Riemannian metric~$g_0$ on $G/H$ such that the length of vectors
which are tangent to fibers of the Riemannian submersion
\begin{equation}\label{fibration}
G/H \to G/K
\end{equation}
is scaled by a constant and the length of vectors normal to the fibers is 
unchanged. That is, we consider the 
 one-parameter family of Riemannian metrics~$g_t$ on~$G/H$, where
$t \in (-\infty,1)$,
\begin{equation}\label{deformation}
g_t (X,Y) :=  \frac{1}{1-t} \cdot g_0(X^{\m},Y^{\m}) + g_0(X^{\s},Y^{\s}).
\end{equation}
It is well known that the normal homogeneous metric~$g_0$ has nonnegative
curvature. As Schwachh\"{o}fer and Tapp point out in~\cite{ST}, a metric $g_t$
where $t<0$ can be reinterpreted as a submersion metric using Cheeger's
construction~\cite{Ch} in the following way, see also~\cite[Section~2]{Z1}.
Assume $G$ is equipped with the biinvariant metric~$g_0$ and the homogeneous
space~$K/H$ is equipped with the metric $\lambda^2 \cdot \left. g_0 \right|_{\k
\times \k}$. Then an isometric action of~$K$ on $G \times K/H$ is defined by $k
\cdot (g, aH) = (gk^{-1},kaH)$. The quotient by this action is diffeomorphic
to~$G/H$ and the submersion metric corresponds to (\ref{deformation}) with $t =
-1/\lambda^2$. Since the metric on the total space of this submersion has
nonnegative curvature, it follows from O'Neill's formula that the metric on the
base space has nonnegative curvature.

However, metrics $g_t$ with $t>0$ do not have an such an interpretation and
whether they have nonnegative curvature depends on the triple $(H,K,G)$.
Schwachh\"{o}fer and Tapp prove the following condition on the Lie algebra level.

\begin{theorem}{\rm \cite{ST}}\label{ThmST}
\begin{enumerate}

\item[(1)] The metric $g_t$ has nonnegative curvature for small $t > 0$ if
    and only if there exists some $C > 0$ such that for all X and Y in
    $\p$,
\begin{equation}\label{condition}
\left|[X^{\m}, Y^{\m}]^{\m}\right| \leq C\left|[X,Y]\right|.
\end{equation}

\item[(2)] In particular, if $(K, H)$ is a symmetric pair, then $g_t$ has
    nonnegative curvature for small $t > 0$, and in fact for all $t \in
    (-\infty,1/4]$.

\end{enumerate}
\end{theorem}
The first part of~(2) follows immediately from the observation that
condition~(\ref{condition}) always holds when $[\m,\m] \subseteq \h$, in which
case the left hand side of the inequality is obviously zero for all $X$ and $Y$
in~$\p$. In particular, (\ref{condition}) holds if $\k$ is abelian.

When $H$ is trivial the submersion~(\ref{fibration}) becomes $K \to G \to G/K$
and $g_t$ is in fact a left-invariant metric on $G$.  In this case,
Schwachh\"ofer \cite{Sch} proved $g_t$ has nonnegative curvature for small $t
> 0$ only if the semisimple part of $\k$ is an ideal of $\g$. In particular,
when $\g$ is simple and $\k$ is nonabelian, one does not get nonnegative
curvature by enlarging the fibers.

Here, a pair~$(K,H)$ of compact Lie groups $H\subseteq K$ is said to be a
\emph{symmetric pair} if there exists an automorphism~$\sigma$ of the Lie
algebra $\k$ with $\sigma^2=\id_\k$  whose fixed point set is~$\h$. In this
case we also say that $(\k,\h)$ is a \emph{symmetric pair (of Lie algebras)}.

It is shown in~\cite{ST} that a number of examples $(H,K,G)$ satisfy the
hypothesis of Theorem~\ref{ThmST}. Among them are the following chains $H
\subset K \subset G$, where $[\m,\m] \not\subseteq \h$:
\begin{itemize}
\item $\SU(3) \subset\SU(4) \cong \Spin(6) \subset \Spin(7)$,
\item $\LG_2 \subset \Spin (7) \subset \Spin(p+8)$, where $p \in \{0,1\}$,
    and
\item $\SU(2) \subset \SO(4) \subset \LG_2$, where $\SU(2)$ is contained in
    $\SU(3) \subset \LG_2$.
\end{itemize}

\emph{Notation: } We will denote the Lie algebra of the compact Lie group of
exceptional type $\LG_2$ by $\Lie( \LG_2 )$ in order to avoid confusion with
the notation $\g_1, \g_2, \ldots, \g_m$ in Theorem~\ref{ThmRegular}, by which
we denote the simple ideals of a Lie algebra~$\g$.


\section{A special class of Lie group triples}\label{AClassEx}

In what follows, we will review some general facts about the structure of
\emph{real} simple Lie algebras. Let $T$ be a maximal torus of~$G$ and let $\t$
be its corresponding Lie subalgebra. Consider the adjoint representation $\Ad
\colon G \to \Aut(\g)$ of~$G$ restricted to~$T$. As a $T$-module, the Lie
algebra~$\g$ decomposes into the trivial module~$\t$ and its $g_0$-orthogonal
complement~$\t^{\perp}$, which further decomposes into a sum of pairwise
inequivalent two-dimensional irreducible representations~$\g_{\alpha}^{\R}$
of~$T$ such that for  each element $H \in \t$, $\Ad(\exp(H))$ is the action on
$\g_{\alpha}^{\R}$ by a rotation of the form
\[
\begin{pmatrix}
  \cos( \alpha(H)) &  -\sin( \alpha(H)) \\
  \sin( \alpha(H)) & \;\;\; \cos( \alpha(H)) \\
\end{pmatrix}
\]
with respect to a suitable basis. Via $\Ad(\exp(H)) = \exp(\ad_H)$ we pass to
the Lie algebra level; for each $\g_{\alpha}^{\R}$, we may choose a
$g_0$-orthonormal basis $(X_{\alpha},Y_{\alpha})$ such that
\[
 \ad_H (X_{\alpha}) =  \alpha(H) Y_{\alpha}, \qquad
 \ad_H (Y_{\alpha}) =  -\alpha(H) X_{\alpha}.
\]
Note that $\g_{\alpha}^{\R} = \g_{-\alpha}^{\R}$, in fact, replacing the
ordered basis $(X_{\alpha},Y_{\alpha})$ by $(Y_{\alpha},X_{\alpha})$ reverses
the rotation. Apart from this ambiguity, $\alpha$ is uniquely determined by the
invariant subspace~$\g_{\alpha}^{\R}$. We  refer to the spaces
$\g_{\alpha}^{\R}$ as the \emph{real root spaces} of~$\g$. Let $R$ be the set
of nonzero elements $\alpha \in \t^*$ such that there exists a nontrivial
$T$-module $\g_{\alpha}^{\R}$ in $\t^{\perp}$. Now choose a vector $v \in \t$
such that $\alpha(v) \neq 0$ for all $\alpha \in R$. Let $R_+ := \{ \alpha \in
R \mid \alpha(v) > 0 \}$. We say $R_+$ is the set of \emph{positive roots}
of~$\g$ and $R = R_+ \cup -R_+$ is the set of \emph{roots}. Let $\g^{\C}= \g
\otimes \C$ denote the complexified Lie algebra. For any $\alpha \in \t^*$ one
defines
\[
\g^{\C}_{\alpha} =
 \left \{ X \in \g^{\C} \mid \ad_H (X) = i \alpha(H) X \mbox{ for all } H \in \t \right \}.
\]
Whenever $\g^{\C}_{\alpha} \neq 0$, we say $\g^{\C}_{\alpha}$ is a \emph{root
space} of~$\g$. For each $\alpha \in R_+$, define the following elements of
$\g^{\C}$:
\[
E_{ \alpha} = X_{\alpha} - i Y_{\alpha}, \qquad
E_{-\alpha} = X_{\alpha} + i Y_{\alpha},
\]
so that we have $\ad_H (E_{\pm \alpha}) = i\alpha(H) E_{\pm \alpha}$. In this
way we obtain the \emph{root space decomposition}
\[
\g^{\C} = \left( \t \otimes \C \right) \oplus \bigoplus_{\alpha \in R} \g^{\C}_{\alpha}.
\]
 Of course $\g^{\C}_{\alpha} \neq 0$ if and only if $\alpha \in R$. In this case
 $\g^{\C}_{\alpha} = \C E_{\alpha}$.
\begin{remark}
For each $\alpha$, the basis vectors $X_{\alpha},Y_{\alpha}$ are 
determined only up to rotation: for any $t \in \R$  we may replace $E_{\alpha}$ and
$E_{-\alpha}$  by $e^{it}E_{\alpha}$ and $e^{-it}E_{-\alpha}$, which amounts to
replacing $X_\alpha$ and $Y_\alpha$ by
\begin{align}\label{EqRotateXY}
   \cos(t) X_{\alpha} + \sin(t) Y_{\alpha}  \quad {\rm and} \quad
  - \sin(t) X_{\alpha} + \cos(t) Y_{\alpha}.
\end{align} \end{remark}

It follows from the Jacobi identity that $ [ \g_{\alpha}^{\C}, \g_{\beta}^{\C}]
\subseteq \g_{\alpha + \beta}^{\C}$ for all $\alpha,\beta \in \t^*$. In
particular, if $\alpha + \beta$ is not a root, $[ \g_{\alpha}^{\C},
\g_{\beta}^{\C} ]=0$. On the other hand, we have $[ \g_{\alpha}^{\C},
\g_{\beta}^{\C} ] = \g_{\alpha + \beta}^{\C}$ whenever $\alpha + \beta \neq 0$
\cite[Thm.~4.3, Ch.~III]{Hel}.  We define  the number $N_{\alpha,\beta}$ by $[
E_{\alpha}, E_{\beta} ] = N_{\alpha,\beta} E_{\alpha + \beta}$ if $\alpha,
\beta, \alpha+\beta \in R$.

Let $\Gamma$ be a subset of $R$. Then $\Gamma$ is said to be \emph{symmetric}
if whenever $\alpha \in \Gamma$, then $-\alpha \in \Gamma$. The set $\Gamma$ is
said to be \emph{closed} if whenever $\alpha,\beta \in \Gamma$ and
$\alpha+\beta\in R$, then $\alpha+\beta \in \Gamma$.

\bigskip


We will now consider examples which arise from a special class of compact
symmetric spaces~$G/K$, namely those for which $\rk(G/K) = \rk(G)$, i.e.\ $\s$
contains a maximal abelian subalgebra of~$\g$. The simply connected irreducible
spaces of this type are $\SU(n) / \SO(n)$, $\SO(2n{+}1) / \SO(n{+}1) \times
\SO(n)$, $\SO(2n) / \SO(n) \times \SO(n)$, $\Sp(n) / \U(n)$, $\LE_6 / \Sp(4)$,
$\LE_7 / \SU(8)$, $\LE_8 / \SO'(16)$, $\LF_4 / \Sp(3)\Sp(1)$, $\LG_2 / \SO(4)$,
cf.\ \cite[Table V, Ch.X]{Hel}. For the spaces in this class, the Satake
diagram of~$G/K$ is the same as the Dynkin diagram of~$G$, but with uniform
multiplicity one. The corresponding involution~$\sigma$ on~$\g$ induces an
involution of~$\g^{\C}$ which acts as minus identity on~$\t \otimes \C$ and
sends each root to its negative. Furthermore, every real root space
$\g_{\alpha}^{\R}$ is $\sigma$-invariant,
 with a one-dimensional fixed point set. For each basis $(X_{\alpha},Y_{\alpha})$
 we have
\begin{equation}\label{Sigma}
\sigma(X_{\alpha}) = X_{\alpha}, \quad
\sigma(Y_{\alpha}) = - Y_{\alpha}
\end{equation}
after an appropriate rotation as in (\ref{EqRotateXY}). In particular, for the
subalgebra~$\k$, $\k={\rm span}\{X_{\alpha} \mid \alpha \in R_+\}$ and
for its 
complement, 
$\s = \t \op {\rm span}\{Y_{\alpha}\mid \alpha \in R_+\}$.

Assume $\h$ is spanned by a subset of $\{X_{\alpha}\,|\,\alpha \in R_+\}$ and
$(K,H)$ is not a symmetric pair. Then the $g_0$-orthogonal complement~$\m$
of~$\h$ in~$\k$ contains two elements of the form $X_{\lambda}$, $X_{\mu}$ such
that the $\m$-component of their bracket $[X_{\lambda},X_{\mu}]^{\m}$ is
nonzero (this implies $\lambda+\mu\neq 0$). Since $[X_{\lambda},X_{\mu}] \in
\g_{\lambda+\mu}\op \g_{\lambda-\mu} \op \g_{\mu-\lambda} \op
\g_{-\lambda-\mu}$, we may assume (possibly after interchanging $\lambda$ and
$\mu$) that at least one of $\lambda\pm\mu$ is a positive root of $\g$ for
which $X_{\lambda\pm\mu}\in\m$. Let $\nu=\lambda\pm\mu$, where the sign is
chosen such that $\nu \in R_+$ and $X_{\nu} \in \m$. The set of roots which
appear as nonzero linear combinations of $\lambda$ and $\mu$ with integral
coefficients forms an irreducible rank two root system~$R(\lambda,\mu)$ which
is of type $\LA_2$ or $\LB_2$. (No root system of type~$\LG_2$ ever occurs as a
proper subset of an irreducible root system.)

We will now show that we may choose two linearly independent elements
$\alpha,\beta \in \{\lambda, \mu, \nu\}$ such that $\alpha-\beta\in R_+$ and
neither $\alpha+\beta$ nor $2\alpha-\beta$ is a root. In case $R(\lambda,\mu)$
is of type $\LA_2$, choose $\alpha := \lambda$, $\beta := \mu$ if $\nu =
\lambda-\mu$ and choose $\alpha := \nu$, $\beta := \mu$ if $\nu = \lambda+\mu$.
Assume $R(\lambda,\mu)$ is of type $\LB_2$ and $\nu = \lambda + \mu$. Then
either $\lambda$ and $\mu$ are orthogonal short roots, in which case we set
$\alpha := \nu$, $\beta := \mu$, or $\lambda$ and $\mu$ are of different length
and enclose an angle of~$\frac{3\pi}4$, in which case we choose $\alpha := \nu$
and $\beta \in \{\lambda, \mu\}$ to be the long root.

Now assume $R(\lambda,\mu)$ is of type $\LB_2$ and $\nu = \lambda - \mu$. Then
either $\lambda$ and $\mu$ are orthogonal short roots, in which case we set
$\alpha := \lambda$, $\beta := \nu$, or $\lambda$ and $\mu$ are of different
length and enclose an angle of~$\frac\pi4$, in which case we choose $\alpha \in
\{\lambda,\mu\}$ to be the long root and $\beta := \nu$.

For any $H \in \t$ and for any real constant~$\eta$, define $X = X_{\alpha} +
H$ and $Y = X_{\beta} + \eta Y_{\alpha-\beta}$. We show that $[X^\m,Y^\m]\neq
0$, and  $[X,Y]=0.$  We recall,
\[
[H,X_{\beta}] = \beta(H) Y_{\beta}, \quad
[H,Y_{\alpha-\beta}] = -(\alpha(H)-\beta(H)) X_{\alpha-\beta}.
\]
Furthermore, using the fact that $\alpha+\beta$ is not a root of~$\g$, we
obtain
\begin{align*}
4[X_{\alpha},X_{\beta}]
 &= [E_{\alpha} + E_{-\alpha}, E_{\beta} + E_{-\beta}]
 = N_{\alpha,-\beta} E_{\alpha-\beta} + N_{-\alpha,\beta} E_{\beta-\alpha}\\
 &= N_{\alpha,-\beta} (X_{\alpha-\beta}-iY_{\alpha-\beta})
  +N_{-\alpha,\beta} (X_{\alpha-\beta}+iY_{\alpha-\beta}).
\end{align*}
It follows that $N_{\alpha,-\beta} = N_{-\alpha,\beta} \in \R \setminus \{0\}$
and we have shown that
\[
[X^\m,Y^\m]=[X_{\alpha},X_{\beta}] = \frac12 N_{\alpha,-\beta} X_{\alpha-\beta}.
\]
Similarly, we calculate, using the fact that $2\alpha-\beta$ is not a root
of~$\g$,
\begin{align*}
4[X_{\alpha},Y_{\alpha-\beta}]
 &= i[E_{\alpha} + E_{-\alpha}, E_{\alpha-\beta} - E_{\beta-\alpha}]
 = -iN_{\alpha,\beta-\alpha} E_{\beta} + iN_{-\alpha,\alpha-\beta} E_{-\beta}\\
 &= -N_{\alpha,\beta-\alpha} (iX_{\beta}+Y_{\beta})
  +N_{-\alpha,\alpha-\beta} (iX_{\beta}-Y_{\beta}).
\end{align*}
Since $[X_{\alpha},Y_{\alpha-\beta}] \in \s$ it follows that
$N_{\alpha,\beta-\alpha}=  N_{-\alpha,\alpha-\beta} \in \R \setminus \{0\}$. We
have shown
\[
[X_{\alpha},Y_{\alpha-\beta}] =  -\frac12 N_{\alpha,\beta-\alpha} Y_{\beta}.
\]
Thus we have
\begin{align*}
[X,Y]
 &= [X_{\alpha} +  H,X_{\beta} + \eta Y_{\alpha-\beta}]=\\
 &= \frac12 N_{\alpha,-\beta} X_{\alpha-\beta}  -\frac12 \eta N_{\alpha,\beta-\alpha}
   Y_{\beta} + \beta(H) Y_{\beta} - \eta (\alpha(H)-\beta(H)) X_{\alpha-\beta}.
\end{align*}
Since $\alpha$ and $\beta$ are linearly independent, there exists an element $H
\in \t$ which solves the equation
\[
\frac12 N_{\alpha,-\beta}N_{\alpha,\beta-\alpha} + 2 \beta(H) (\alpha(H)-\beta(H))= 0.
\]
Having chosen such an $H \in \t$, we set $\eta :=  -2 \beta(H) /
N_{\alpha,\beta-\alpha}$, then
\[
[X,Y]=0
\]
and, on the other hand, $[X^{\m},Y^{\m}]^{\m} = [X_{\alpha},X_{\beta}]^{\m}
\neq 0$. This proves that condition~(\ref{condition}) does not hold.

%

Using Borel-de-Siebenthal theory~\cite[\S3]{Oniscik}, we can describe the
triples $(H,K,G)$ covered by the above calculations more explicitly. Let us
assume we have chosen bases $X_{\alpha}, Y_{\alpha}$ of the real root spaces
of~$\g$ as described above, i.e.\ such that $\k$ is spanned by $\{ X_{\alpha} |
\alpha \in R_+ \}$ and $\h$ is spanned by $\{ X_{\alpha} | \alpha \in S_+ \}$
for some subset $S_+ \subset R_+$. Let $\l$ be the linear subspace of~$\g$
spanned by $\t \cup \{X_{\alpha}|\alpha \in S_+\} \cup \{Y_{\alpha}|\alpha \in
S_+\}$. Consider the complexification $\l^{\C} \subset \g^{\C}$. Then $\l^{\C}
= \left( \t \otimes \C \right) \oplus \bigoplus_{\alpha \in S}
\g^{\C}_{\alpha}$, where $S := S_+ \cup -S_+$. By \cite[Prop.~15, 1
\S3]{Oniscik}, $\l$ is a subalgebra of~$\g$ if and only if the subset $S$ is
closed and symmetric. In our case, $S$ is symmetric by definition. Let us show
that it is also closed: Let $\alpha, \beta \in S$ such that $\alpha + \beta \in
R$. Following~\cite{Wa}, we set $|\alpha| := \alpha$ if $\alpha \in R_+$ and
$|\alpha| := -\alpha$ if $-\alpha \in R_+$. Furthermore, we are using the
convention that whenever $\alpha + \beta \not\in R$ we set $N_{\alpha,\beta}
=0$, $E_{\alpha+\beta}=0$ and whenever $\alpha \not\in R$, we set $X_{\alpha} =
0$. We compute
\begin{align*}
4[X_{\alpha},X_{\beta}]
 &= [E_{\alpha} + E_{-\alpha}, E_{\beta} + E_{-\beta}]\\
 &=   (N_{\alpha,\beta} E_{\alpha+\beta}
    + N_{-\alpha,-\beta} E_{-\beta-\alpha})\\
    &\ \ + (N_{\alpha,-\beta} E_{\alpha-\beta}
    + N_{-\alpha,\beta} E_{-\alpha+\beta}) \\
 &= 2N_{\alpha,\beta} X_{|\alpha+\beta|}
   + 2N_{\alpha,-\beta} X_{|\alpha-\beta|}.
\end{align*}
Since $N_{\alpha,\beta} \neq 0$ and $\h$ is spanned by $\{X_{\alpha}|\alpha\in
S_+\}$, this shows that $X_{|\alpha+\beta|} \in \h$ and hence $\alpha+\beta \in
S$. Thus $\l$ is actually a subalgebra of~$\g$ and $\h = \l \cap \k$. Moreover,
$\rk(\l) = \rk(\g)$ and hence the corresponding subgroup $L$ is closed.

Conversely, let $L \subset G$ be a closed subgroup containing a maximal
torus~$T$ of~$G$. Let $\sigma \colon G \to G$ be an involution as described
above and assume we have chosen a root space decomposition and vectors
$X_{\alpha}$, $Y_{\alpha}$, $\alpha \in R_+$ such that (\ref{Sigma}) holds. Let
$K$ be the fixed point set of~$\sigma$ and let $H := K \cap L$. Then the triple
$H \subset K \subset G$ is such that $\h$ is spanned by $X_{\alpha}$, $\alpha
\in S_+$, for some subset $S_+ \subseteq R_+$. We have proved the following.

\begin{theorem}\label{ThGKmaxRk}
Assume $(G,K)$ is a symmetric pair such that $\rk(G/K) = \rk(G)$ and let
$\sigma \colon G \to G$ be the corresponding involution. Let $L$ be a
$\sigma$-invariant subgroup with $\rk(L) = \rk(G)$, and let $H := K \cap L$.
Then the triple $(H,K,G)$ satisfies condition~(\ref{condition}) if and only if
$(K,H)$ is a symmetric pair.
\end{theorem}

It is interesting to note that while condition~(\ref{condition}) fails for the
triples $H \subsetneq K \subsetneq G$ where $(K,H)$ is not a symmetric pair, it
always holds for the triples $H \subsetneq L \subsetneq G$, since $(L,H)$ is a
symmetric pair. Indeed, the involution $\sigma$ leaves $L$ invariant and thus
induces an involution of the same kind as $\sigma$ (mapping each root to its
negative and acting as minus identity on~$\t$) on $L$.

We do not give a full list of the triples where Theorem~\ref{ThGKmaxRk}
applies, but just illustrate the result by the following examples.

\begin{corollary}\label{ExGKmaxRk}
The following chains $H \subsetneq K \subsetneq G$ of compact Lie groups do not
satisfy condition~(\ref{condition}) in Theorem~\ref{ThmST}~(1):
\begin{enumerate}

\item $\SO(n_1) \times \SO(n_2) \times \SO(n_3) ~\subset~ \SO(n) ~\subset~
    \SU(n)$, $n_i \ge 1$, $n_1 + n_2 + n_3 = n$.

\item $\left[ \SO(n_1+1) \times \SO(n_2) \times \SO(n_3) \right] \times
    \left[ \SO(n_1) \times \SO(n_2) \times \SO(n_3) \right] \\
    ~\subset~ \SO(n+1) \times \SO(n) ~\subset~ \SO(2n+1)$, $n_i \ge 1$,
    $n_1 + n_2 + n_3 = n$.

\item $\U(n_1) \times \U(n_2) \times \U(n_3) ~\subset~ \U(n) ~\subset~
    \Sp(n)$, $n_i \ge 1$, $n_1 + n_2 + n_3 = n$

\item $\left[ \SO(n_1) \times \SO(n_2) \times \SO(n_3) \right] \times
    \left[ \SO(n_1) \times \SO(n_2) \times \SO(n_3) \right] \\
    ~\subset~ \SO(n) \times \SO(n) ~\subset~ \SO(2n)$, where $n_i \ge 1$,
$n_1 + n_2 +
    n_3 = n$.

\item $\SO(3) \cdot \SO(3) \cdot \SO(3) ~\subset~ \Sp(4) ~\subset~ \LE_6$.

\item $\SO(3) \cdot \SO(6) ~\subset~ \SU(8)/\{\pm1\} ~\subset~ \LE_7$.

\item $\SO(3) \cdot \Sp(4) ~\subset~ \SO'(16) ~\subset~ \LE_8$.

\item $\SO(3) \cdot \SO(3) ~\subset~ \Sp(3)\cdot\Sp(1) ~\subset~ \LF_4$.

\end{enumerate}
\end{corollary}

\begin{proof}
All examples are constructed in the following manner. Choose a subgroup $L
\subset G$ of full rank such that $(G,L)$ is not a symmetric pair\footnote{If
$(G,L)$ is a symmetric pair, then $(L,H)$ will be a symmetric pair as well,
since the involution corresponding to $(G,L)$ commutes with $\sigma$.}. Then
determine the subgroup $H \subset L$ (unique up to conjugacy) such that $(L,H)$
is a symmetric pair satisfying $\rk(L/H) = \rk(L)$. Here are the subgroups~$L$
chosen in the examples above: (1) $\eS(\U(n_1)\times\U(n_2)\times\U(n_3))$; (2)
$\SO(2n_1+1)\times\SO(2n_2)\times\SO(2n_3)$; (3)
$\Sp(n_1)\times\Sp(n_2)\times\Sp(n_3)$; (4)
$\SO(2n_1)\times\SO(2n_2)\times\SO(2n_3)$; (5) $\SU(3)\cdot\SU(3)\cdot\SU(3)$;
(6) $\SU(3)\cdot\SU(6)$; (7) $\SU(3)\cdot\LE_6$; (8) $\SU(3)\cdot\SU(3)$. See
\cite[Thm.~16, \S3]{Oniscik} for regular subgroups.
\end{proof}


\section{Subgroups of full rank}
\label{FullRank}

In this section, we consider the case of closed subgroups $H \subsetneq K
\subsetneq G$ of a simple compact Lie group~$G$ such that $\rk(H) = \rk(K) =
\rk(G)$. We show in Theorem~\ref{ThmFullRank} that in this case the triple
$(H,K,G)$ satisfies condition~(\ref{condition}) if and only if $(K,H)$ is a
symmetric pair.  We will show in Theorem~\ref{ThmFullRank} below that we may
restrict ourselves to chains $H \subsetneq K \subsetneq G$ with $\rk(G) \le 3$.
In the following lemma, we prove our result in this special case.

\begin{lemma}\label{LmRankThree}
For the following chains of compact Lie groups $H \subsetneq K \subsetneq G$
there exist elements $X,Y \in \p$ such that $[X,Y]=0$ and
$[X^{\m},Y^{\m}]^{\m}\neq0$.
\begin{enumerate}

\item[(1)] $T^3 ~\subset~ \eS(\U(3)\times\U(1)) ~\subset~ \SU(4),$

\item[(2)] $ \U(2) {\times} \SO(2) ~\subset~ \SO(6) ~\subset~  \SO(7),$

\item[(3)] $ \U(2) {\times} \SO(2) ~\subset~  \SO(5) {\times} \SO(2) ~\subset~  \SO(7),$

\item[(4)] $\U(1) {\times} \U(1) {\times} \U(1) ~\subset~ \U(3) ~\subset~\Sp(3),$

\item[(5a)] $\Sp(1) {\times} \U(1) {\times} \Sp(1) ~\subset~ \Sp(2) {\times} \Sp(1) ~\subset~
    \Sp(3),$

\item[(5b)] $\Sp(1) {\times} \U(1) {\times} \U(1) ~\subset~ \Sp(2) {\times} \U(1) ~\subset~
    \Sp(3),$

\item[(6)] $T^2 ~\subset~ \SU(3) ~\subset~ \LG_2$.

\end{enumerate}
Here $T^n$ denotes an $n$-dimensional torus.
\end{lemma}

\begin{proof}
Case (4) is a special case of part~(3) in Corollary~\ref{ExGKmaxRk}. For all
other cases we exhibit our vectors $X,Y$ using explicit matrix representations.

We identify $\so(n)$ with the set of skew-symmetric real $n {\times} n$-matrices,
$\u(n)$ with the set of Hermitian complex $n {\times} n$-matrices, $\sp(n)$ with the
set of Hermitian quaternionic $n {\times} n$-matrices. Let $E_{\nu\mu}$ denote the
skew-symmetric matrix with the entry~$+1$ in position $(\nu,\mu)$, the
entry~$-1$ in position $(\mu,\nu)$ and zeros elsewhere, while $F_{\nu\mu}$
denotes the symmetric matrix with the entry~$1$ in positions $(\mu,\nu)$ and
$(\nu,\mu)$ and zeros elsewhere.

\begin{enumerate}

\item[(1)] Let $\h =\t^3= \{\diag(it_1,it_2,it_3,-i(t_1+t_2+t_3))\mid
    t_1,t_2,t_3\in \R \}$,
\[
\h \subset \k = \left\{ \left. \left(
               \begin{array}{cc}
                 A &   \\
                   & z \\
               \end{array}
             \right)
\right| A \in \u(3), z = - \tr A \right\}.
\]
We take $X=E_{12} + E_{14}$ and $Y=E_{23} + E_{34}$. Note that $E_{12},
E_{23} \in \m$ while $E_{14}, E_{34}\in \s$. Then $[X^{\m},Y^{\m}] = E_{13}
\in \m$ and $[X,Y]=0$.

\item[(2)] Here
\[
\u(2)=\left\{\begin{pmatrix} X & -Y  \\ Y & X  \end{pmatrix}
~|~Y=Y^t\right\}  \subset \so(4),
\]
and in this way, $\h=\u(2)\op\so(2) \subset \so(4) \op \so(2) \subset \k =
\so(6)$. Thus $\m = \m_1 \op \m_2$ where $\m_1=\so(4)\ominus\u(2)$ and
$\m_2=\so(6)\ominus (\so(4) \op \so(2))$. We take $X = E_{15} - E_{17}$, $Y
= E_{25} + E_{27}$, where
    $E_{15}, E_{25} \in \m$ and $E_{17},E_{27}\in \s$. Then
    $[X^{\m},Y^{\m}] =-E_{12}$ has a nonzero $\m$ component, while
    $[X,Y]=0$.

\item[(3)] We have $\h=\u(2) \oplus \so(2) \subset \so(4) \oplus \so(2)
    \subset \so(5) \oplus \so(2)$, and thus $\m = \m_1 \op \m_2$ where
    $\m_1=\so(5) \ominus \so(4)$, and $\m_2=\so(4)\ominus\u(2)$. We take $X
    = E_{15}+E_{16}$ and $Y=E_{25}-E_{26}$, where $E_{15},~E_{25} \in
    \m_1$, and $E_{16},~E_{26} \in \s$.  Then $[X^{\m},Y^{\m}] = -E_{12}$,
    which has a nonzero $\m$-component (in $\m_2$), and $[X,Y]=0$.

\item[(4)] See Corollary~\ref{ExGKmaxRk}~(3).

\item[(5)] In both cases (5a) and (5b), we have the same $\m=\m_1\op\m_2$
    where $\m_1=\fsp(2)\ominus (\fsp(1)\op\fsp(1))$ and
    $\m_2=\fsp(1)\ominus\u(1)$.
We may take $X^{\m} =iF_{12}$ and $Y^{\m} =kF_{12}$ (both in $\m_2$). We
take $X^{\s} = E_{23} +iF_{13}$ and $Y^{\s} = jF_{23} -kF_{13}$ (in both
cases).
    Then $[X^{\m},Y^{\m}]=-2j(F_{11}-F_{22})$, so that
    $[X^{\m},Y^{\m}]^{\m}=2jF_{22} \neq 0$, while $[X,Y]=0$.

\item[(6)] See \cite[Subsection~2.4 (1c)]{kk}.

\end{enumerate}
\end{proof}

\begin{remark}\label{RemPrimeTriple}
Let $H \subsetneq K \subsetneq G$ be a triple of compact Lie groups and suppose
$H' \subsetneq K' \subsetneq G'$ is another triple of compact Lie groups with
$G' \subseteq G$ such that for the orthogonal complement~$\p'$ of $\h'$ in
$\g'$ we have $\p' \subset \p$, while for the orthogonal complement~$\m'$ of
$\h'$ in $\k'$ we have $\m' \subset \m$. Then it is sufficient to exhibit a
pair of vectors $X,Y \in \p'$ such that $[X,Y]=0$ but $[X^{\m'},Y^{\m'}]^{\m'}
\neq 0$ in order to show that the triple $(H,K,G)$ does not satisfy
condition~(\ref{condition}): Since $X^{\m}=X^{\m'}$, $Y^{\m}=Y^{\m'}$, and
$[X^{\m'}, Y^{\m'}]\in\g'$, we know $[X^{\m'}, Y^{\m'}]^{\m'} =
[X^{\m},Y^{\m}]^{\m}$.
\end{remark}

\begin{remark}\label{RemEnlargeH}
Conversely, with the notation as in the remark above, if the triple $(H,K,G)$
satisfies condition~(\ref{condition}) then the triple $(H',K',G')$ also
satisfies condition~(\ref{condition}). In particular, if the group $H$ is
enlarged to $H'$, condition~(\ref{condition}) is preserved.
\end{remark}

In the following, we will consider the case of compact Lie groups $H \subsetneq
K \subsetneq G$ where $H$ and $K$ are closed subgroups of full rank in~$G$ such
that $(K,H)$ is not a symmetric pair. We may chose a maximal torus~$T$ of~$H$,
which is then also maximal torus of $K$ and $G$, and consider a root space
decomposition with respect to~$T$. Using the notation of
Section~\ref{AClassEx}, there are subsets $R_H \subsetneq R_K \subsetneq R$,
such that
\[\h^{\C} = \t^{\C} \oplus \bigoplus_{\alpha \in R_H} \g_{\alpha}, \qquad
\k^{\C} = \t^{\C} \oplus \bigoplus_{\alpha \in R_K} \g_{\alpha},
\]
where $\h^{\C}, \k^{\C}, \t^{\C}, \m^\C, \s^\C$ denote the complexifications of
$\h,\k,\m,\s,\t$, respectively. Let $R_{\m} = R_K \setminus R_H$ and $R_{\s} =
R \setminus R_K$.  The scalar product on $\t^*$ is induced from~$g_0$. The root
systems $R_H$ and $R_K$ of $H$ and $K$, respectively, are symmetric, hence the
subsets $R_{\m}$ and $R_{\s}$ are also symmetric,
see~\cite[Ch.\,1,\,\S3.11]{Oniscik}.

\begin{theorem}\label{ThmFullRank}
Let $G$ be a simple compact Lie group and let $H \subsetneq K \subsetneq G$ be
closed subgroups. If $\rk(H) = \rk(K) = \rk(G)$ then either $(K,H)$ is a
symmetric pair or there exist elements $X,Y \in \p$ such that $[X,Y]=0$ and
$[X^{\m},Y^{\m}]^{\m}\neq0$.
\end{theorem}

\begin{proof}
Since $(K,H)$ is not a symmetric pair, there exist $\lambda, \mu \in R_{\m}$
such that $[\g_\lambda, \g_\mu] \not \subset \h^\C$. Hence, $0 \neq
\g_{\lambda+\mu} \subset \m^\C$. Let $R(\lambda,\mu)$ denote the subset of~$R$
consisting of all roots which are nonzero linear combinations of $\lambda$ and
$\mu$ with integer coefficients. Because $R(\lambda,\mu)$ contains the six
roots $\g_{\pm\lambda}, \g_{\pm\mu}, \g_{\pm(\lambda+\mu)}$, it is an
irreducible root system of rank two.

Since $\g$ is simple, $\k^\C$ acts effectively on~$\s^\C$ and there is a root
$\nu \in R_{\s}$ such that $[\g_{\lambda},\g_{\nu}]\neq0$. Let
$R(\lambda,\mu,\nu)$ be the set of all roots which are nonzero linear
combinations of $\lambda,\mu,\nu$ with integer coefficients. Then
$R(\lambda,\mu,\nu)$ is a closed, symmetric subsystem of~$R$ and hence a root
system of rank two or three, which contains $R(\lambda,\mu)$ as a closed,
symmetric, proper subsystem.

We know $\lambda + \nu \in R$, thus the root system $R(\lambda,\mu,\nu)$ is
irreducible. Hence the inclusion $R(\lambda,\mu) \subset R(\lambda,\mu,\nu)$ is
of one of the following types
\begin{equation}\label{inclusions}
\begin{array}{lll}
 (1)\quad\LA_2 \subset \LA_3,\qquad
&(2)\quad\LA_2 \subset \LB_3,\qquad
&(3)\quad\LB_2 \subset \LB_3, \\
 (4)\quad\LA_2 \subset \LC_3,\qquad
&(5)\quad\LB_2 \subset \LC_3,\qquad
&(6)\quad\LA_2 \subset \LG_2.
\end{array}
\end{equation}
(Maximal closed symmetric subsets of irreducible root systems are given in
\cite[Ch.1, \S3.11]{Oniscik}.)

Let $\g'$ be the semisimple part of the subalgebra of~$\g$ spanned by the real
root spaces $\g_{\alpha}^\R$, $\alpha \in R(\lambda,\mu,\nu)$ and~$\t$. Then
$\g'$ is in fact a simple Lie algebra isomorphic to $\su(4)$, $\so(7)$,
$\sp(3)$, or $\Lie( \LG_2 )$. Let $\h' := \h \cap \g'$ and $\k' := \k \cap
\g'$. Then $\m' := \m \cap \g'$ is the orthogonal complement of~$\h'$ in~$\k$
and $\s' := \s \cap \g'$ is the orthogonal complement of~$\k'$ in $\g'$. In
particular, Remark~\ref{RemPrimeTriple} applies.

We complete the proof by showing that for each type of inclusion
$R(\lambda,\mu) \subset R(\lambda,\mu,\nu)$ as enumerated
in~(\ref{inclusions}), the triple $\h' \subsetneq \k' \subsetneq \g'$
corresponds to one of the triples in Lemma~\ref{LmRankThree}. Thus there exist
$X,Y \in \p' \subseteq \p$ such that $[X,Y]=0$ and $[X^{\m'},Y^{\m'}]^{\m'} =
[X^\m,Y^\m]^\m \neq 0$.

Since $\g_\nu \subset \s$ we have that $\k' \subsetneq \g'$ and since
$\g_{\pm\lambda}, \g_{\pm\mu}, \g_{\pm(\lambda+\mu)} \subset \m$, we know
$[\m',\m'] \not\subseteq \h$. In particular $\m' \neq 0$. Let $\k^*$ be the
subalgebra of~$\g$ spanned by~$\t$ and the real root spaces $\g_\alpha^\R$,
$\alpha \in R(\lambda,\mu)$. Since $\lambda, \mu \in R_\m$, it follows that
$\k^*$ is contained in $\k'$.

When the subalgebra~$\k^* \subsetneq \g'$ is maximal, it follows that $\k^* =
\k'$.  This is the case for all possible inclusions $\k^* \subset \g'$
enumerated in (\ref{inclusions}) except (2) and (5). In case~(2) we have $\LA_2
\subset \LA_3 \subset \LB_3$, thus we further distinguish the cases (2a)~$\k'
\cong \so(6)$ and (2b)~$\k' = \k^* \cong \u(3)$. In case~(5) we distinguish the
cases (5a)~$\k' \cong \sp(2) \oplus \sp(1)$ and (5b)~$\k' = \k^* \cong \sp(2)
\oplus \u(1)$.

In the cases (1), (4) and (6) where the semisimple part of $\k'$ is of type
$\LA_2$, it follows that $\h'$ is abelian, as all six roots of $R(\lambda,\mu)$
are contained in $R_\m$; hence the triple $\h' \subset \k' \subset \g'$ is
determined, up to an automorphism of~$\g'$. We have one of the chains of Lie
groups  in Lemma~\ref{LmRankThree}, (1), (4), or (6).

If the semisimple part of $\k'$ is of type $\LB_2$ as in cases (3) and (5),
then either six or eight roots of $R(\lambda,\mu)$ are contained in $R_\m$ and
it follows that $\h$ is either abelian or its semisimple part is of type
$\LA_1$. Note that there are two possibilities for the inclusion $\LA_1 \subset
\LB_2$, corresponding to the inclusions of Lie groups $\SU(2) \subset \SO(5)$
and $\SO(3) \subset \SO(5)$; in the first case the triple $\h' \subset \k'
\subset \g'$ corresponds to the chains (3) or (5) in Lemma~\ref{LmRankThree}.
In the second case, the triple $\h' \subset \k' \subset \g'$ would correspond
to either $\SO(3) \times \SO(2) \times \SO(2) \subset \SO(5) \times \SO(2)
\subset \SO(7)$ in case~(3) or $\U(2) \times \Sp(1) \subset \Sp(2) \times
\Sp(1) \subset \Sp(3)$ in case~(5). But in both cases $(K,H)$ is a symmetric
pair, a contradiction.

In case~(2a) we have $\g' \cong \so(7)$ and $\k' \cong \so(6)$. Hence $\h'
\subset \k'$ is a subalgebra of rank~3 such that $(\k',\h')$ is not a symmetric
pair. The only possibilities are $\h' \cong \u(2) \oplus \u(1)$ and $\h' \cong
\so(2) \oplus \so(2) \oplus \so(2)$. The first case is covered by
Lemma~\ref{LmRankThree}~(2), the second case by Remark~\ref{RemPrimeTriple} and
Lemma~\ref{LmRankThree}~(2).

Finally, in case~(2b) we have that $\k' \cong \u(3)$, it follows that $\h$ is
abelian, since $\m'$ is at least $6$-dimensional. Hence the triple $\h' \subset
\k' \subset \g'$ corresponds to a triple $T^3 \subset \U(3) \subset \Spin(7)$.
Using Remark~\ref{RemPrimeTriple} once more to replace $\Spin(7)$ by
$\Spin(6)$, we see that this case is covered by Lemma~\ref{LmRankThree}~(1) via
the isomorphism $\SU(4) \cong \Spin(6)$.

We have now shown that in each case there exist elements $X,Y \in \p$ such that
$[X,Y]=0$ and $[X^{\m},Y^{\m}]^{\m}\neq0$.
\end{proof}

\begin{corollary}\label{CorFullRank}
Let $G$ be a compact Lie group and let $H \subsetneq K \subsetneq G$ be closed
subgroups. If $\rk(H) = \rk(K) = \rk(G)$ then  the triple satisfies
condition~(\ref{condition}) if and only if for each simple factor~$G_i$ of~$G$,
at least one of the following holds.
\begin{enumerate}
\item[(i)] $(G_i \cap K, G_i \cap H)$ is a symmetric pair.
\item[(ii)] $\g_i \subseteq \k$.
\item[(iii)] $\g_i \cap \k \subseteq \h$.
\end{enumerate}
In particular, if there is a simple factor~$G_i$ of~$G$ for which none of the
above conditions holds, then there exist elements $X,Y \in \p$ such that
$[X,Y]=0$ and $[X^{\m},Y^{\m}]^{\m}\neq0$.

\end{corollary}

\begin{proof}
Let $\g=\z\oplus\g_1 \oplus \ldots \oplus \g_m$, where $\z$ is the center of
$\g$ and where the $\g_i$ are the simple factors of~$\g$. Since $\rk(H) =
\rk(K) = \rk(G)$, we have that $\k=\z \oplus (\g_1 \cap \k) \oplus \ldots
\oplus (\g_m \cap \k)$ and $\h=\z \oplus (\g_1 \cap \h) \oplus \ldots \oplus
(\g_m \cap \h)$. From this fact it is obvious that condition~(\ref{condition})
holds if and only if it holds for each triple  $(H \cap G_i, K \cap G_i, G_i)$
where $i =1, \ldots, m$. Now the first part of the corollary follows from
Theorem~\ref{ThmFullRank}.

Assume there is a simple factor~$G_i$ of~$G$ for which none of (i), (ii), (iii)
above holds. Apply Remark~\ref{RemPrimeTriple} to the chain $H \cap G_i
\subsetneq K \cap G_i \subsetneq G_i$ to see that the second part of the
assertion follows.
\end{proof}


\section{New examples}
\label{NewExamples}

\begin{theorem}\label{ThmSUSOSO}
The chain $\SU(n) \subset \SO(2n) \subset \SO(2n+1)$, $n \ge 2$ satisfies
condition~(\ref{condition}).
\end{theorem}

\begin{proof}
Let $n \ge 2$, $\g = \so(2n+1)$  and $\k = \so(2n)$. We introduce a complex
structure on $\R^{2n}$ by
\begin{equation}\label{EqJ}
J = \begin{pmatrix}
      \fbox{$\begin{smallmatrix}0 & -1\\1&\ \ 0 \end{smallmatrix}$} &&&\\
      &\fbox{$\begin{smallmatrix}0 & -1\\1&\ \ 0 \end{smallmatrix}$} &&\\
      &&\ddots&\\
      &&&\fbox{$\begin{smallmatrix}0 & -1\\1&\ \ 0 \end{smallmatrix}$} \\
    \end{pmatrix}.
\end{equation}
Then $\u(n) = \{ A \in \so(2n) \mid JA = AJ \},$ the subset of all matrices in
$\so(2n)$ which commute with~$J$. The orthogonal complement $\m_0$ of $\u(n)$
in $\so(2n)$ is then $\m_0 = \{ M \in \so(2n) \mid JMJ=M \},$ the subset of all
matrices in $\so(2n)$ which anticommute with~$J$. The one-dimensional ideal
which is the center of $\u(n)$ is $\R J =: \z$, that is, $\u(n) = \su(n) \oplus
\z$. The orthogonal complement of $\h$ in $\k$ is $\m = \m_0 \oplus \z$. For
the orthogonal complement $\s$ of $\k$ in~$\g$, we have $[\s,\s] \subseteq \k$.

Let $X^\m, Y^\m \in \m$ and $X^\s, Y^\s \in \s$. Define $X := X^\m +X^\s$, $Y
:= Y^\m +Y^\s$ in $\p = \m  \oplus \s$.
Let $W \in \so(2n)$ be a matrix of rank two and let $x \in \R^{2n}$ be a unit
vector such that $x \in (\ker W)^\perp$. Let $\lambda :=  |W(x)|$ and let $y :=
W(x) / \lambda$. Since $W$ is skew symmetric, we have $\im (W) = (\ker
W)^\perp$ and $W(v) \perp v$ for all $v \in \R^{2n}$. Hence $(\ker W)^\perp$ is
spanned by the orthonormal vectors $x$ and $y$. It follows that
\[
W = \lambda(yx^t -xy^t).
\]
Define $S(x,y) := yx^t -xy^t$ for $x,y \in \R^{2n}$. We have seen that any rank
two matrix in $\so(2n)$ is given by $S(x,y)$ for some $x,y \in \R^{2n}$. In
particular, we have $[X^\s,Y^\s] = S(x,y)$  for some $x,y \in \R^{2n}$.

For $t \in \R$ we may define the pair $(\tilde X,\tilde Y) \in \p \times \p$ by
\[
 \tilde X := \cos(t)X + \sin(t)Y, \quad
 \tilde Y := \cos(t)Y - \sin(t)X.
\]
Then we have $ [\tilde X,\tilde Y] = [X,Y] $ and $[\tilde X^\m,\tilde Y^\m]^\m
= [X^\m,Y^\m]^\m$. In particular, by choosing $t$ suitably, we have $\tilde
Y^\z =  \cos(t)Y^\z - \sin(t)X^\z=0$. Dropping the tildes, we may therefore
assume that $ Y^\z = 0 $ and hence $Y^\m = Y^{\m_0}$.

Recall that $\SO(2n) / \U(n)$ is a symmetric space of rank~$r := \lfloor \frac
n2 \rfloor$, see~\cite{Hel}. Indeed, conjugation by the matrix $J$ defines an
involutive automorphism of~$\k = \so(2n)$ whose $(+1)$-eigenspace is $\h \oplus
\z$ and whose $(-1)$-eigenspace is $\m_0$. Let
\[
D := \begin{pmatrix}
  0 & 0 & 1 & 0 \\
  0 & 0 & 0 & -1 \\
  -1 & 0 & 0 & 0 \\
  0 & 1 & 0 & 0 \\
\end{pmatrix}.
\]
Consider the maximal abelian subalgebra $\La \subset \m_0$ defined as follows.
For $n$~even, define $\La$ to be the space of all skew symmetric real $2n
\times 2n$-matrices of the form
\begin{equation}\label{EqMaxAbEven}
\begin{pmatrix}
t_1D & & & \\
 & t_2D & & \\
 & & \ddots & \\
 & & &  t_rD \\
\end{pmatrix}\end{equation}
where $t_1, \ldots, t_r \in \R$. If $n$ is odd, define $\La$ to be the space of
all skew symmetric real $2n \times 2n$-matrices of the form
\begin{equation}\label{EqMaxAbOdd}
\left(\begin{array}{cccc|cc}
 t_1D & & & & & \\
 & t_2D & & & &\\
 & & \ddots & &  & \\
 & & & t_rD && \\ \hline
& & & & 0 & 0\\
& & & & 0 & 0\\
\end{array}\right),
\end{equation}
where $t_1, \ldots, t_r \in \R$. Since the isometric action of the group $H$
on~$\g$ given by restriction of the adjoint representation of~$G$ leaves
$\m_0$, $\z$ and $\s$ invariant, we may replace $(X,Y)$ by $(\Ad_{h}(X),
\Ad_{h}(Y))$ for any $h \in H$ without limitation of generality, since we have
\begin{align*}
\left| [\Ad_{h}(X),\Ad_{h}(Y)] \right|
= \left|\Ad_h([X,Y])\right|= \left|[X,Y]\right|
\end{align*}
and
\begin{align*}
\left|[\Ad_{h}(X)^\m,\Ad_{h}(Y)^\m]^\m\right| &=
\left|[\Ad_{h}(X^\m),\Ad_{h}(Y^\m)]^\m\right| \\ &=
\left|\Ad_{h}([X^\m,Y^\m])^\m\right| \\ &= \left|[X^\m,Y^\m]^\m\right|.
\end{align*}
It follows from the theory of symmetric spaces~\cite{Hel} that the subspace
$\La$ intersects all orbits of the $H$-action on~$\m_0$, in particular, there
is an element $h \in H$ such that $\Ad_h(Y^\m) \in \La$ and we may henceforth
assume $Y^\m \in \La$, i.e. $Y^\m$ is of the form (\ref{EqMaxAbEven}) or
(\ref{EqMaxAbOdd}). We define a subalgebra of $\k$ isomorphic to $r \cdot
\so(4)$ as follows. If $n$ is even, define
\[
\k_1 := \left \{ \left.
\begin{pmatrix}
A_1 & & & \\
 & A_2 & & \\
 & & \ddots & \\
 & & & A_r \\
\end{pmatrix} \in \so(4r)
\right| A_1, \ldots, A_r \in \so(4)
\right\}.
\]
If $n$ is odd, define
\[
\k_1 := \left \{ \left.
\left(\begin{array}{cccc|cc}
A_1 & & & & & \\
 & A_2 & & & &\\
 & & \ddots & &  & \\
 & & & A_r && \\ \hline
& & & & 0 & 0\\
& & & & 0 & 0\\
\end{array}\right) \in \so(4r+2)
\right| A_1, \ldots, A_r \in \so(4)
\right\}.
\]
Let $P \colon \k \to \k_1$ be the orthogonal projection from $\k$ onto $\k_1$.
Note that $X^\z, Y^\m \in \k_1$.  Furthermore, the action of $\k_1$ on $\k$
leaves $\k_1$ and its orthogonal complement invariant, hence we have
$P([V,Y^\m]) = [P(V),Y^\m]$ for all $V \in \k$. We have
\begin{align}\label{EqBE}
\left| [X,Y] \right| & \ge
\left| [X,Y]^\k \right| \ge
\left| P ([X,Y]^\k) \right| =
\left| P ([X^\m,Y^\m]  + [X^\s,Y^\s]) \right| = \nonumber \\
&=\left| P ([X^\z,Y^\m]  + [X^{\m_0},Y^\m]  + [X^\s,Y^\s]) \right| =\nonumber  \\
&=\left| [X^\z,Y^\m]  + [P(X^{\m_0}),Y^\m]  + P([X^\s,Y^\s]) \right|
\end{align}
Now let $P_1, \ldots, P_r \colon \k \to \so(4)$ be the orthogonal projections
onto the direct summands of $\k_1$ isomorphic to $\so(4)$ such that
$P_{\ell+1}$ maps a matrix $A=(a_{ij}) \in \k$ to its $4 \times 4$-submatrix
\begin{equation*}
\begin{pmatrix}
  a_{4\ell+1,\, 4\ell+1} & \dots & a_{4\ell+1,\, 4\ell+4} \\
  \vdots &  & \vdots \\
  a_{4\ell+4,\, 4\ell+1} & \dots & a_{4\ell+4,\, 4\ell+4} \\
\end{pmatrix}.
\end{equation*}
Let $J_1 = P_1(J)=\ldots=P_r(J)$ be the complex structure on $\R^4$ defined by
(\ref{EqJ}) in the case $n=2$. From (\ref{EqBE}) we have
\begin{align}\label{EqBE2}
\left| [X,Y]^\k \right|^2 & \ge \sum_{\nu=1}^r   \left| [cJ_1,P_\nu(Y^m)] +
 [P_\nu(X^{\m_0}),P_\nu(Y^\m)] + P_\nu([X^\s,Y^\s]) \right|^2.
\end{align}
Since $P_\nu([X^\s,Y^\s])$ is a matrix of rank two or zero, we have
$P_\nu([X^\s,Y^\s]) = [U^{\s'},V^{\s'}] = S(x,y)$ for two vectors $x
=(x_1,x_2,x_3,x_4),\ y=(y_1,y_2,y_3,y_4) \in \R^4.$ We claim that there is a
constant $C>0$ such that \[\left| [X,Y]^\k \right|^2 \ge C^2 \cdot \left|
[X^\m,Y^\m]^\m \right|^2\] for all $X,Y \in \p$. Each of the $r$ summands on
the right hand side in (\ref{EqBE2}) is of the form
\[
 \left| [U^{\z'},V^{\m'}] +  [U^{\m_0'},V^{\m'}] + [U^{\s'},V^{\s'}] \right|^2 =
\left| [U,V]^{\k'} \right|^2
\]
where $U,V \in \p' = \m'+\s'$ is a pair of vectors for the triple
\[
(H',K',G')=(\SU(2),\SO(4),\SO(5))
\]
where $\g'=\k'\oplus\s'$ and $\k'=\h'\oplus\m'$ are orthogonal decompositions
and where $\z$ is spanned by $J_1$. Hence it suffices to verify the claim for
the case $n=2$. We have $U^{\z'}=cJ_1$ and $V^{\m'} = \tau D$ for some $\tau
\in \R$. Consequently $[U^{\z'},V^{\m'}]  = t E$, where
\[
E :=
\begin{pmatrix}
  0 & 0 & 0 & 1 \\
  0 & 0 & 1 & 0 \\
  0 & -1 & 0 & 0 \\
  -1 & 0 & 0 & 0 \\
\end{pmatrix},
\]
and $t=2c\tau$. Furthermore, using the fact that $\m_0'$ is spanned by $D$ and
$E$, it is easy to verify that for an arbitrary matrix $U^{\m_0'}\in\m_0'$ we
have $[U^{\m_0'},V^{\m'}] =sF$, where
\[
F :=
\begin{pmatrix}
  0 & 1 & 0 & 0 \\
  -1 & 0 & 0 & 0 \\
  0 & 0 & 0 & 1 \\
  0 & 0 & -1 & 0 \\
\end{pmatrix}
\]
and $s \in \R$. After choosing the vectors $U^{\z'},U^{\m_0'},V^{\m'}$, the
real numbers $t$ and $s$ are uniquely determined. Define $W :=
[U^{\s'},V^{\s'}] = $
\[
 \left( \begin {array}{cccc} 0&{ y_1}\,{ x_2}-{ y_2}\,{
 x_1}&{ y_1}\,{ x_3}-{ y_3}\,{ x_1}&{ y_1}\,{
 x_4}-{ y_4}\,{ x_1}\\\noalign{\medskip}{ y_2}\,{
x_1}-{ y_1}\,{ x_2}&0&{ y_2}\,{ x_3}-{ y_3}\,{
 x_2}&{ y_2}\,{ x_4}-{ y_4}\,{ x_2}
\\\noalign{\medskip}{ y_3}\,{ x_1}-{ y_1}\,{ x_3}&{
 y_3}\,{ x_2}-{ y_2}\,{ x_3}&0&{ y_3}\,{ x_4}-
{ y_4}\,{ x_3}\\\noalign{\medskip}{ y_4}\,{ x_1}-{
y_1}\,{ x_4}&{ y_4}\,{ x_2}-{ y_2}\,{ x_4}&{
y_4}\,{ x_3}-{ y_3}\,{ x_4}&0\end {array} \right)
\]
We will finish the proof by showing that there is a constant $C>0$ such that
\[
|[U,V]^{\k'}|^2 \ge C^2 |[U^{\m'},V^{\m'}]^{\m'}|^2 .
\]
Using the Euclidean scalar product given by $\langle A, B \rangle = \tr(A^tB)$
and the corresponding  norm on $\R^{4 \times 4}$, we have
\[
 \left|[U^{\m'},V^{\m'}]^{\m'}\right|^2 = \left| tE+sF \right|^2 = 4(s^2+t^2).
\]
Assume that $y=(y_1,y_2,y_3,y_4) \in \R^4$ is a fixed unit vector. We choose
the orthonormal basis
\[
y,\, e_1 :=(-y_4,-y_3,y_2,y_1),\, e_2 :=  (-y_2,y_1,-y_4,y_3),\, e_3 :=(y_3,-y_4,-y_1,y_2)
\]
of $\R^4$, and write $x = a_0 y + a_1 e_1 + a_2 e_2 + a_3 e_3$ with $a_j \in
\R$. The squared length of $W$ is $ \tr((yx^t-xy^t)^t(yx^t-xy^t)) = 2
\left|x\right| \left|y\right| - 2\langle x,y\rangle = 2(a_1^2+a_2^2+a_3^2). $
The length of the orthogonal projection of $W$ on the linear subspace spanned
by $E$ is \[\left| y_1x_4 -y_4x_1+y_2x_3-y_3x_2 \right|=
\left|(-y_4,-y_3,y_2,y_1)(x_1,x_2,x_3,x_4)^t \right| = \left|a_1\right|.\] The
length of the orthogonal projection of $W$ on the linear subspace spanned by
$F$ is \[\left| y_1x_2 -y_2x_1+y_3x_4-y_4x_3 \right| = \left| (-y_2,
y_1,-y_4,y_3)(x_1,x_2,x_3,x_4)^t \right| = \left| a_2 \right|.\] We have, since
$|sE|=2s$ and $|tF|=2t$,
\begin{align*}
\left|[U,V]^{\k'}\right|^2 &=a_1^2+a_2^2+2a_3^2
 +\left(a_1\pm 2s\right)^2 +\left( a_2\pm 2t\right)^2 \\
&= 2a_1^2+2a_2^2+2a_3^2 \pm 4a_1s \pm 4a_2t +4s^s +4t^2 \\
&=2a_3^2 +2\left(a_1 \pm s \right)^2 +2\left(a_2 \pm t \right)^2 +
2s^2 +2t^2  \\
&\ge 2 (s^2 +t^2)  = \frac12 \left| sE + tF \right|^2
 = \frac12 \left| [U^{\m'},V^{\m'}]^{\m'}\right|^2.
\end{align*}
This finishes the proof.
\end{proof}


\section{Regular subgroups}
\label{Regular}

A closed subgroup~$K$ of a compact Lie group~$G$ is called a \emph{regular
subgroup} if $\rk (C_G(K)) = \rk(G) - \rk(K) + \rk(Z(K)),$ where $C_G(K)$ is
the centralizer of~$K$ in~$G$ and $Z(K)$ is the center of~$K$.  In this case,
we also call the Lie algebra~$\k$ of~$K$ a \emph{regular subalgebra} of the Lie
algebra~$\g$ of~$G$. When $K \subseteq G$ is a regular subgroup,  there exists
a maximal torus~$T$ of~$G$ with Lie algebra~$\t$ such that $\k$ is spanned by a
subset of~$\t$ and the real root spaces $\g_{\alpha}^\R$, $\alpha \in S$, where
$S$ is a symmetric and closed subsystem of the root system~$R$ of~$G$. Consider
a chain of compact Lie groups $H \subsetneq K \subsetneq G$ where $K$ is a
regular subgroup of~$G$. Then $H$ is regular subgroup of $G$ if and only if it
is a regular subgroup of~$K$.

\begin{lemma}\label{LmRegular}
Let $H \subsetneq K \subsetneq G$ be compact Lie groups such that $G$ is simple
and $H$, $K$ are regular subgroups. Let $T_1$ be a maximal torus of~$C_K(H)$.
Assume that $(K, H \cdot T_1)$ is not a symmetric pair. Then there exist
elements $X,Y \in \p$ such that $[X,Y]=0$ and $[X^{\m},Y^{\m}]^{\m}\neq0$.
\end{lemma}

\begin{proof}
Let $T_2$ be a maximal torus of~$C_G(K)$. Let $H' := H \cdot T_1 \cdot T_2$,
let $K' := K \cdot T_2$, let $G' := G$. Then $\rk(H') = \rk(K') = \rk(G')$. Now
apply Corollary~\ref{CorFullRank} and Remark~\ref{RemPrimeTriple} to the chain
$H' \subsetneq K' \subsetneq G'$.
\end{proof}

\begin{proposition}\label{PropHermitian}
Let $K$ be a simple compact Lie group and let $H \subset K$ be such that
$C_K(H)$ is of positive dimension and $(K,H \cdot T_1)$ is a symmetric pair for
a maximal torus $T_1 \subseteq C_K(H)$. Then $K / (H \cdot T_1)$ is a Hermitian
symmetric space and $T_1$ is one-dimensional.
\end{proposition}

\begin{proof}
This follows from the classification of symmetric spaces \cite{Hel}.
\end{proof}

\begin{remark}\label{RemNotInZ}
Let $H$, $K$ and $T_1$ be as in Proposition~\ref{PropHermitian}. Let $\k = \h
\oplus \t_1 \oplus \m_0$ be an orthogonal decomposition and let $\m=\t_1 \oplus
\m_0$. (It follows that $H \subset K$ is a regular subgroup~\cite{Hel}.) Let
$T_0$ be a maximal torus of $H$. Then $T := T_0 \cdot T_1$ is a maximal torus
of~$K$. Consider a root space decomposition of~$\k$ and let $R_K$ denote the
set of roots. Let $R_H$ denote the subsystem corresponding to the subgroup~$H$.
The set $R_H$ consists of all roots in $R_K$ which vanish on~$\t_1$. Hence it
follows that the projection of $[X_\beta,Y_\beta]$ on $\t_1$ is non-zero for
all $\beta \in R_\m := R_K \setminus R_H$.
\end{remark}

\begin{proposition}\label{PropRegular}
Let $G$ be a compact Lie group. Let $H \subsetneq K \subsetneq G$ be connected
compact Lie groups such that $H$, $K$ are regular subgroups. Assume that for
each simple ideal~$I$ of~$\g$ the condition~(\ref{condition}) holds for the
triple of Lie algebras $(I \cap \h, I \cap \k, I)$. Then (\ref{condition}) also
holds for $(H,K,G)$.
\end{proposition}

\begin{proof}
Let $T_2$ be a maximal torus of $C_G(K)$. Let $\g = \g_0 \oplus \g_1 \oplus
\ldots \oplus \g_n$ be a decomposition into ideals of~$\g$ such that $\g_0$ is
abelian and $\g_1, \ldots, \g_n$ are simple. Define $\k_i := \g_i \cap \k$ and
$\h_i := \g_i \cap \h$. Let $\m_i$ be the orthogonal complement of~$\h_i$
in~$\k_i$ and let $\p_i$ be the orthogonal complement of~$\h_i$ in~$\g_i$. Now
assume that $(\h_i, \k_i, \g_i)$ satisfies~(\ref{condition}) with a positive
constant~$C_i$, i.e.\ we have
\[
\left|[X^{\m_i}, Y^{\m_i}]^{\m_i}\right| \leq C_i \left|[X,Y]\right|.
\]
for $i = 1, \ldots, n$. Set $C := \max(C_1,\ldots,C_n)$. Now let $X,Y \in \p$,
where $\m$ is the orthogonal complement of~$\h$ in~$\k$ and where $\p$ is the
orthogonal complement of~$\h$ in~$\g$. Define $\h' := \h_0 + \ldots + \h_n$ and
let $\p' = \p_0 + \ldots + \p_n$ be the orthogonal complement of~$\h'$ in $\g$.
Since $\h' \subseteq \h$, it follows that $\p' \supseteq \p$ and hence $X,Y \in
\p'$. Let $\k' := \k_0 + \ldots + \k_n$ and let $\m' := \m_0 + \ldots + \m_n$.
We have $\m \subseteq \m' + \t_2$. Since $T_2 \subseteq C_G(K)$, it follows
that $ [ X^\m , Y^\m ] = [ X^{\m'} , Y^{\m'} ]. $ We write $X = X_0 + \ldots +
X_n$ and $Y = Y_0 + \ldots + Y_n$, where $X_i,Y_i \in \g_i$ and compute
\begin{align*}
|[ X^\m , Y^\m ]^\m| &\le |[ X^{\m'} , Y^{\m'} ]^{\m'}| =
|\sum_{i=0}^n [ X^{\m_i} , Y^{\m_i} ]^{\m_i}| \le\\&\le
|\sum_{i=1}^n C_i [ X_{i} , Y_{i} ]| \le
| \sum_{i=1}^n C [ X_{i} , Y_{i} ]| = C |[X,Y]|.\tag*{\qedhere}
\end{align*}
\end{proof}

We will now prove our classification result for chains of regular subgroups.

\begin{theorem}\label{ThmRegular}
Let $G$ be a compact Lie group. Let $H \subsetneq K \subsetneq G$ be connected
compact Lie groups such that $H$, $K$ are regular subgroups of~$G$. If the
triple $(H,K,G)$ satisfies condition~(\ref{condition}) then for each simple
ideal~$\g_i$ of~$\g$ one of the following is true.
\begin{enumerate}
\item \label{trivials} $\g_i \cap \k = \g_i$, i.e.\  the simple ideal
    $\g_i$ is contained in~$\k$.

\item \label{sympair} $\g_i \cap \k \neq \g_i$ and $(\g_i \cap \k, \g_i
    \cap \h)$ is a symmetric pair, possibly such that $\g_i \cap \k$ is
    contained in~$\h$.

\item \label{sochain} $\g_i \cong \so(2n+1)$, $\g_i \cap \k \cong \so(2n)$
    and $\g_i \cap \h \cong \su(n)$.

\item \label{spchain} $\g_i \cong \sp(n)$ where each but one simple ideal
    of $\g_i \cap \k$ is contained in $\h$ and the one simple ideal not
    contained in $\h$ is isomorphic to~$\sp(1)$, standardly embedded.

\item \label{gtwochain} $\g_i \cong \Lie( \LG_2 )$, $\g_i \cap \k \cong
    \so(4)$ and $\g_i \cap \h \cong \su(2)$ such that $\g_i \cap \h$ is
    contained in a subalgebra $\su(3) \subset \g_i$.
\end{enumerate}
\end{theorem}

\begin{proof}

Assume (\ref{condition}) holds for the triple $(H,K,G)$. Let $T_0$ be a maximal
torus of $H$. Let $T_1$ be a maximal torus of $C_K(H)$ and let $T_2$ be a
maximal torus of $C_G(K)$. Let $R$ denote the root system of $G$ with respect
to $T := T_0 \cdot T_1 \cdot T_2$. Generalizing the notation from
Section~\ref{AClassEx}, let $R_H$ and $R_K$ denote the closed symmetric subsets
of $R$ corresponding to the full rank subgroups $H \cdot T_1 \cdot T_2 \subset
G$ and $K \cdot T_2 \subset G$, respectively. Let $R_\m = R_K \setminus R_H$
and $R_\s = R \setminus R_K$.

Consider the decomposition $K = T_K \cdot K_1 \cdot \ldots \cdot K_n$  where
$T_K$ is a torus and where the $K_i$ are simple. Let $G=Z \cdot G_1 \cdot
\ldots \cdot G_m$, where $Z$ is the center of~$G$ and where $G_1,\ldots,G_m$
are the simple factors of~$G$. Since $K$ is a regular subgroup, there is a map
$f \colon \{1,\ldots,n\} \to \{1,\ldots,m\}$ such that $\k_i \subseteq
\g_{f(i)}$.

Let $\RR$ be the set of roots $\beta \in R_\m$ for which there is an $i \in
\{1,\dots,n\}$ such that $\g_{\beta}^\R \subset \k_i$, $\k_i \neq \g_{f(i)}$
and $(K_i,K_i \cap H)$ is not a symmetric pair. If the set $\RR$ is empty then
one of the first two conditions in the statement of the theorem holds for
each~$\g_i$. Thus we may assume the set $\RR$ is non-empty.

If $\beta \in \RR$ and $\g_{\beta}^\R \subset \k_i$, then we may apply
Lemma~\ref{LmRegular} and Proposition~\ref{PropHermitian} to the chain $\h \cap
\k_i \subsetneq \k_i \subsetneq \g_{f(i)}$, showing that the pair $(K_i, H \cap
K_i)$ is as described in Proposition~\ref{PropHermitian}.

Since $G_{f(i)}$ is simple and $\k_i \neq \g_{f(i)}$, it follows that for each
$\beta \in \RR$ there is at least one root $\alpha \in R_\s$ such that $\alpha
+ \beta \in R$ or $\alpha - \beta \in R$. Replacing $\beta$ with $-\beta$, if
necessary, we may assume that $\alpha + \beta \in R$. Let $\PP$ be the set of
all pairs $(\alpha,\beta)$ such that $\beta \in \RR$, $\alpha \in R_\s$ and
$\alpha + \beta \in R$. For each pair $(\alpha,\beta) \in \PP$, consider the
set $R(\alpha,\beta)$ of all linear combinations of $\alpha$ and $\beta$ with
integer coefficients. This set $R(\alpha,\beta)$ is a closed symmetric
subsystem of~$R$, thus the elements of $R(\alpha,\beta)$ form a root system of
rank two. Since $\beta \in R_\m$, $\alpha, \alpha+\beta \in R_\s$, it follows
that the root system $R(\alpha,\beta)$  is irreducible, thus of type $\LA_2$,
$\LB_2$, or $\LG_2$.

First assume that among all elements of~$\PP$ there is at least one pair
$(\alpha,\beta)$ such that $R(\alpha,\beta)$ is of type $\LA_2$. Then we have
$R(\alpha,\beta) = \{ \pm \alpha, \pm \beta, \pm (\alpha + \beta) \}$. Let
$\g'$ be the subalgebra of $\g$ generated by the vectors $X_\alpha$,
$Y_\alpha$, $X_\beta$, $Y_\beta$, $X_{\alpha+\beta}$, $Y_{\alpha+\beta}$. The
algebra~$\g'$ is isomorphic to $\su(3)$. We may choose a three-dimensional
complex representation of $\g'$ and assume that $X_\alpha = E_{12}, Y_\alpha =
iF_{12}, X_\beta = E_{23}, Y_\beta = iF_{23}, X_{\alpha+\beta} = E_{13},
Y_{\alpha+\beta} = iF_{13}$. Define $X^{\m'} = E_{23}$, $Y^{\m'} = iF_{23}$,
$X^{\s'} = E_{12}+E_{13}$, $Y^{\s'} = i(F_{12}-F_{13})$. We see that $[X^{\m'}
+X^{\s'},Y^{\m'} + Y^{\s'}]=0$, while $[X^{\m'},Y^{\m'}] = 2i(F_{22}-F_{33})$.
By Remark~\ref{RemNotInZ} this contradicts the assumption that the triple
$(H,K,G)$ satisfies (\ref{condition}).

Thus we may assume there is no pair $(\alpha,\beta) \in \PP$ such that
$R(\alpha,\beta)$ is of type $\LA_2$. Hence each such $R(\alpha,\beta)$ is of
type $\LB_2$ or $\LG_2$.

Assume $(\alpha, \beta) \in \PP$ is such that $R(\alpha,\beta)$ is of
type~$\LG_2$. Since the root system~$\LG_2$ does not occur as a proper
subsystem of any irreducible root system, it follows that the subalgebra $\g'$
generated by the vectors $X_\lambda$, $Y_\lambda$, $\lambda \in
R(\alpha,\beta)$ is a simple ideal of~$\g$ isomorphic to~$\Lie( \LG_2 )$. It
follows from Lemma~\ref{LmRegular} and Proposition~\ref{PropHermitian} that the
chain $\h \cap \g' \subsetneq \k \cap \g' \subsetneq \g'$ is as described in
part~(\ref{gtwochain}) of the theorem, since otherwise we find a pair of
commuting vectors $X, Y \in \g'$, such that $[X^\m,Y^\m]^\m \neq 0$ by the
results of~\cite[Subsection~2.4]{kk}.

Finally, we may assume that all root systems $R(\alpha,\beta)$, $(\alpha,\beta)
\in \PP$, are of type $\LB_2$.

First assume that there is at least one pair $(\alpha, \beta) \in \PP$ such
that $R(\alpha,\beta)$ is of type $\LB_2$ and such that $\beta$ is a short
root. Since $R_K \cap R(\alpha,\beta)$ is a closed symmetric subsystem of
$R(\alpha,\beta)$ it follows that $\pm \beta \in R_\m$, while $\lambda \in
R_\s$ for all $\lambda \in R(\alpha,\beta) \setminus \{\pm \beta\}$. Let $\g'$
be the subalgebra of $\g$ generated by all vectors $X_\lambda, Y_\lambda$ where
$\lambda \in R(\alpha,\beta)$ and let $\k' = \k \cap \g'$. The algebra $\g'$ is
isomorphic to~$\so(5)$ and $\k'$ is contained in the regular subalgebra $\so(3)
\oplus \so(2)$, it is isomorphic to either $\so(3)$ or $\so(3) \oplus \so(2)$.
We may choose a representation of $\g'$ on~$\R^5$ such that $E_{ij}$ with $i =
1,2$, $j = 3,4,5$ represent elements of $\s$ and $E_{23}, E_{45} \in \t$.
Define $X^{\m} = E_{12}$ and $X^{\s} = -E_{24}$, $Y^{\m} = E_{13}$ and $Y^{\s}
=E_{34}$. Then $[X,Y] = [E_{12} -E_{24}, E_{13}+E_{34}]=0$, yet
$[X^{\m},Y^{\m}] = [E_{12},E_{13}]=-E_{23}$. It follows from
Remark~\ref{RemNotInZ} that the triple $(H,K,G)$ does not satisfy
(\ref{condition}).

Now assume that for all root systems $R(\alpha,\beta)$, $(\alpha,\beta) \in
\PP$ which are of type $\LB_2$ each $\beta$ is a long root. Then there are the
following alternatives. Either among these there is at least one pair $(\alpha,
\beta) \in \PP$ such that there is a long root $\lambda \in R(\alpha,\beta)
\cap R_\s$ or all long roots in $R(\alpha,\beta)$ are contained in $R_K$ for
each $(\alpha, \beta) \in \PP$ such that $R(\alpha,\beta)$ is of type~$\LB_2$.

Consider the first case.  Let $\g'$ be the subalgebra of $\g$ generated by all
vectors $X_\lambda, Y_\lambda$ where $\lambda \in R(\alpha,\beta)$ and let $\k'
= \k \cap \g'$. The algebra $\g'$ is isomorphic to $\so(5)$ and $\k'$ is one of
the subalgebras $\su(2)$ or $\u(2)$ of $\g'$. Choose a representation of $\g'$
on~$\R^5$ such that the subalgebra $\u(2)$ is represented by the linear
combinations of the matrices $E_{23}, E_{24}+E_{35}, E_{34}-E_{25},E_{45}$.
Define $X = X^{\m} + X^{\s}$, $Y = Y^{\m} + Y^{\s}$ where $X^{\m} =
\tfrac12(E_{25} +E_{34})$ and $Y^{\m}=\tfrac12(E_{23}+E_{45})$. Take $X^{\s} =
E_{14} +\tfrac12(E_{23}-E_{45})$ and $Y^{\s} =  E_{12}
+\tfrac12(E_{25}-E_{34})$. Computing, we see that $[X,Y] = 0$ while
 $[X^{\m},Y^{\m}]  = -\tfrac12(E_{24}-E_{35})$. It follows from
Remark~\ref{RemNotInZ} that the triple $(H,K,G)$ does not satisfy
(\ref{condition}).

Finally we are left with the case where all root systems $R(\alpha,\beta)$,
$(\alpha,\beta) \in \PP$ are of type $\LB_2$, such that each $\beta$ is a long
root and all long roots in $R(\alpha,\beta)$ are contained in $R_K$ for each
$(\alpha, \beta) \in \PP$. Fix one such root system $R(\alpha_0,\beta_0)$. Let
$\g'$ be the simple ideal of~$\g$ which contains the vectors $X_{\beta_0},
Y_{\beta_0}$. Then $\g'$ is a regular subalgebra of~$\g$. We have to show that
$\g'$ is as described in~(\ref{sochain}) or in~(\ref{spchain}). Since
$R(\alpha_0,\beta_0)$ is of type~$\LB_2$, it follows that the simple factor
of~$G$ corresponding to $\g'$ is of type $\LB_n$, $\LC_n$, or $\LF_4$.

Now the statement of the theorem follows from Lemma~\ref{LmRoots}. Indeed, in
case $\g'$ is of type $\LB_n$, it follows from Lemma~\ref{LmRoots} that the
simple ideal $\g'$ is as in item~(\ref{sochain}) of the theorem. In case $\g'$
is of type $\LF_4$, it follows from the Lemma that $\g'$ is as described in
item~(\ref{sympair}) of the theorem. If $\g'$ is of type $\LC_n$, we use the
following counterexample to show that there can only be one simple ideal of $\k
\cap \g'$ which is not contained in~$\h$. Consider the chain of Lie algebras
$\{0\} \subset \sp(1) \oplus \sp(1) \subset \sp(2)$. It follows from
\cite[Lemma~2.2]{Sch} that a pair of vectors $X,Y$ with $[X,Y]=0$, while
$[X^\m,Y^\m]^\m = [X^\m,Y^\m] \neq 0$ exists. Indeed, define e.g.\
\begin{align*}
X = \left(
      \begin{array}{cc}
        j & j+1 \\
        j-1 & i-k \\
      \end{array}
    \right), \qquad
Y = \left(
      \begin{array}{cc}
        -i-k & i \\
        i & \frac i2+j-\frac k2 \\
      \end{array}
    \right)
\end{align*}
Then $[X,Y]=0$, while $[X^\m,Y^\m]^\m = [X^\m,Y^\m] \neq 0$. Using
Remark~\ref{RemPrimeTriple}, it follows that $\g'$ is as described in
item~(\ref{spchain}) of the theorem.
\end{proof}

To solve the remaining cases in the proof of the above theorem, we have to take
a closer look at the root systems of type $\LB_n$, $\LC_n$ and $\LF_4$.

\begin{lemma}\label{LmRoots}
Let $G$ be a simple compact Lie group of type $\LB_n$, $\LC_n$, or $\LF_4$. Let
$H \subsetneq K \subsetneq G$ be connected subgroups such that $\rk(H) = \rk(K)
= \rk(G)$ and such that $(K,H)$ is a Hermitian symmetric pair. Then one of the
following is true, where $R_\s$ and $R_\m$ are defined as in
Section~\ref{AClassEx}.
\begin{enumerate}

\item \label{SameL} There is pair of non-orthogonal roots $(\alpha,\beta)
    \in R_\s \times R_\m$ of the same length.

\item \label{LongShort} There is pair of non-orthogonal roots
    $(\alpha,\beta) \in R_\s \times R_\m$ such that $\alpha$ is a long root
    and $\beta$ is a short root.

\item \label{SoSo} The triple of Lie algebras $(\h,\k,\g)$ is isomorphic to
    \[(\u(n),\so(2n),\so(2n+1)).\]

\item \label{SpSp} The triple $(\h,\k,\g)$ is of the following form: $\g
    \cong \sp(n)$ and for each of the simple factors $\k_1, \ldots, \k_f$
    of $\k$ we have either $\k_j \subseteq \h$ or $\k_j \cong \sp(1)$ and
    $\k_j \cap \h \cong \u(1)$.

\item\label{FSpin} The triple $(\h,\k,\g)$ is isomorphic to
    $(\so(8),\so(9),\Lie (\LF_4))$.

\end{enumerate}
\end{lemma}

\begin{proof}
Assume first $G$ is of type $\LB_n$, i.e.\ $\g \cong \so(2n+1)$. The Lie
algebra $\k$ is contained in a maximal subalgebra~$\k'$ of full rank of $\g$.
This maximal subalgebra $\k'$ is conjugate to $\so(2\ell) \oplus
\so(2n-2\ell+1)$, where $\ell =1,\ldots,n$, \cite[Thm.~16, \S3]{Oniscik}. Using
the notation of \cite[Prop.~6.5, Ch.~V]{BtD}, the long roots of~$G$ are given
by $\pm \vartheta_\mu \pm \vartheta_\nu$ where $1 \le \mu<\nu \le n$, while the
short roots are $\pm \vartheta_\nu$ where $1 \le \nu \le n$. Then we may assume
\[
R_{K'} = \{ \pm \vartheta_\mu \pm \vartheta_\nu \mid 1 \le \mu < \nu \le \ell
\vee \ell < \mu < \nu \le n \} \cup \{ \pm \vartheta_\nu  \mid \ell < \nu \le n
\}.
\]
and we have for the roots $R_{\s'} \subseteq R_\s$, corresponding to the
orthogonal complement $\s'$ of $\k'$ in $\g$,
\[
R_{\s'} = \{ \pm \vartheta_\mu \pm \vartheta_\nu \mid 1 \le \mu \le \ell < \nu
\le n \} \cup \{ \pm \vartheta_\nu  \mid 1 \le \nu \le \ell \}.
\]
We may assume $\ell < n$ since otherwise (\ref{SoSo}) holds. Since $H \neq K$,
there is an element $\beta \in R_\m \subseteq R_{K'}$. If $\beta$ is a long
root, say $\beta = \pm \vartheta_\mu \pm \vartheta_\nu$, then we may choose an
element $\alpha = \pm \vartheta_\kappa \pm \vartheta_\lambda \in R_{\s'}
\subseteq R_\s$ where $\kappa=\mu$ or $\kappa=\nu$ and (\ref{SameL}) holds. If
$\beta$ is a short root $\pm \vartheta_\nu$ for some $\ell < \nu \le n$, then
choose $\alpha =  \vartheta_1 + \vartheta_\nu$ to show that (\ref{LongShort})
holds.

Now assume $G$ is of type $\LC_n$, i.e.\ $\g \cong \sp(n)$, $n \ge 3$. The Lie
algebra $\k$ is again contained in maximal subalgebra~$\k'$ of full rank of
$\g$. The maximal subalgebras $\k'$ of full rank in $\sp(n)$ are conjugate to
$\u(n)$, or $\sp(\ell) \oplus \sp(n-\ell)$, where $\ell =1,\ldots,\lfloor \frac
n2 \rfloor$, see \cite[Thm.~16, \S3]{Oniscik}. It follows by induction that all
simple factors of $K$ are of type $\LA_k$ or $\LC_k$, $k < n$. Now, using the
notation from \cite[Prop.~6.6, Ch.~V]{BtD}, the short roots of~$G$ are $\pm
\vartheta_\mu \pm \vartheta_\nu$ where $1 \le \mu<\nu \le n$ and the long roots
are $\pm 2 \vartheta_\nu$ where $1 \le \nu \le n$. Assume first $\k$ is
contained in a maximal subalgebra of~$\g$ conjugate to $\u(n)$. Then we may
assume
\[
R_{K'} = \{\vartheta_\mu - \vartheta_\nu \mid 1 < \nu, \mu \le n,\; \mu \neq
\nu \}.
\]
and we have for the roots $R_{\s'} \subseteq R_\s$, corresponding to the
orthogonal complement $\s'$ of $\k'$ in $\g$,
\[
R_{\s'} = \{ \pm (\vartheta_\mu + \vartheta_\nu) \mid 1 \le \mu < \nu \le n, \}
\cup \{ \pm 2\vartheta_\nu  \mid 1 \le \nu \le n \}.
\]
Since $H \neq K$, there is an element $\beta \in R_\m \subseteq R_{K'}$.
Obviously, there is some element $\alpha \in R_{\s'}$ which is not orthogonal
to~$\beta$. Since $R_{K'}$ contains only short roots, either (\ref{SameL}) or
(\ref{LongShort}) holds.

Assume now $\k$ is not conjugate to a subalgebra of~$\u(n)$, hence contained in
a maximal subalgebra of~$\g$ conjugate to $\sp(\ell) \oplus \sp(n-\ell)$.  We
may assume
\[
R_{K'} = \{ \pm \vartheta_\mu \pm \vartheta_\nu \mid 1 \le \mu < \nu \le \ell
\vee \ell < \mu < \nu \le n \} \cup \{ \pm 2\vartheta_\nu  \mid 1 \le \nu \le n
\}.
\]
and we have for the roots $R_{\s'} \subseteq R_\s$, corresponding to the
orthogonal complement $\s'$ of $\k'$ in $\g$,
\[
R_{\s'} = \{ \pm \vartheta_\mu \pm \vartheta_\nu \mid 1 \le \mu \le \ell < \nu
\le n \}.
\]
Assume that (\ref{SpSp}) does not hold. Then there is a simple ideal $\k_j$
of~$\k$ which is either isomorphic to $\u(q)$ for some $q \in \{1, \ldots,
n-1\}$ or isomorphic to some $\sp(q)$ for some $q \in \{2, \ldots, n-1\}$ and
such that $\k_j \cap \h \neq \k_j$. It follows from \cite[\S3]{Oniscik} that
then $R_\m$ contains a short root. But for any root~$\gamma$ the set $R_{\s'}$
contains a (short) root which is not orthogonal to~$\gamma$. Thus (\ref{SameL})
holds.

Now assume $G$ is of type $\LF_4$. The long roots are $\pm \vartheta_\mu \pm
\vartheta_\nu$, $1 \le \mu < \nu \le 4$ and the short roots are given by either
$\pm \vartheta_\mu$, $1 \le \mu \le 4$ or $\frac12 (\pm \vartheta_1 \pm
\vartheta_2 \pm \vartheta_3 \pm \vartheta_4)$, see \cite{bourbaki} or
\cite{Oniscik}. The maximal subgroups of maximal rank in $\LF_4$ are
$\Spin(9)$, $\Sp(3) \cdot \Sp(1)$, and $\SU(3) \cdot \SU(3)$, see
\cite[Thm.~16, \S3]{Oniscik}.

Assume first that $K$ is conjugate to a subgroup of $\SU(3) \cdot \SU(3)$. We
deduce from \cite[Thm.~16, \S3]{Oniscik} that the root system corresponding to
this subgroup consists of six short roots and six long roots. The 24 long roots
from $R$ comprise a subsystem of type $\LD_4$. Thus for any element $\gamma$ in
$R$ there is a long root in $R_{\s'}$ which is non-orthogonal to $\gamma$.
Hence (\ref{LongShort}) holds.

Now assume $K$ is conjugate to a subalgebra of $\Sp(3) \cdot \Sp(1)$. The root
system of this group consists of 12~short and 8~long roots. These 8~long roots
correspond to a subalgebra of type $4 \cdot \sp(1) \cong \so(4) \oplus \so(4)
\subset \so(8)$. A similar argument as above shows that (\ref{LongShort})
holds.

Finally, assume $K$ is conjugate to a subgroup of $\Spin(9)$. We have
\[
R_{K'} = \{ \pm \vartheta_\mu \pm \vartheta_\nu \mid 1 \le \mu < \nu \le 4\}
\cup \{ \pm \vartheta_\mu \mid 1 \le \mu \le 4 \}.
\]
If $K$ is any maximal connected subgroup of maximal rank in $\Spin(9)$ then
(\ref{LongShort}) holds, cf.\ \cite[Thm.~16, \S3]{Oniscik}. Thus we may assume
$K = \Spin(9)$. Note that for every root~$\gamma$ in $R$ there is a (short)
root in the set $R_{\s'} = \{ \textstyle \frac12 (\pm \vartheta_1 \pm
\vartheta_2 \pm \vartheta_3 \pm \vartheta_4) \} $ which is not orthogonal
to~$\gamma$. Thus if (\ref{SameL}) does not hold then $R_\m$ consists
exclusively of long roots. It follows that $H = \Spin(8)$.
\end{proof}

\begin{remark}\label{ThmSpChain}
For items (1), (2) (3) and (5) in Theorem~\ref{ThmRegular} we know that
condition~(\ref{condition}) holds for the chains $(\h \cap \g_i, \k \cap \g_i,
\g_i)$. However, we conjecture that condition~(\ref{condition}) holds also for
each chain of regular subgroups $(H,K,G) = (\Sp(1)^{n-1}, \Sp(1)^n, \Sp(n))$
with $n \ge 2$. If the conjecture is true, it follows from
Proposition~\ref{PropRegular} that the statement in Theorem~\ref{ThmRegular}
can be improved to ``if and only if''. To our knowledge, there are no known
examples of chains $(H,K,G)$ satisfying condition~(\ref{condition}) which
contain non-regular subgroups, cf.~\cite{kk}.
\end{remark}



\begin{thebibliography}{9999}


\bibitem[B]{bourbaki} N.\ Bourbaki, \emph{Lie groups and Lie algebras.
    Chapitres 4, 5 et 6}. \'{E}l\'{e}ments de Math\'{e}matique, Hermann (1968); english
    translation: \emph{Lie groups and Lie algebras. Chapters 4--6}. Translated
    from the 1968 French original by Andrew Pressley. Elements of Mathematics,
    Springer (2002)

\bibitem[BtD]{BtD} T.~Br\"{o}cker, T.\  tom Dieck, \emph{Representations of compact
    Lie groups}, Graduate Texts in Mathematics, \textbf{98}, Springer (1985)

\bibitem[Ch]{Ch} J.~Cheeger, \emph{Some examples of manifolds of nonnegative
    curvature}, J.~ Diff.~Geom.~\textbf{8} (1973) 623--628.

%
%
%
\bibitem[GZ]{GZ} K.~Grove, W.~Ziller, \emph{Curvature and symmetry of Milnor
    spheres}, Ann.\ of Math.\ (2) \textbf{152} (2000) 331--367.

\bibitem[H]{Hel} S.\ Helgason, \emph{Differential geometry, Lie groups and
    symmetric spaces}. Academic Press (1978)


\bibitem[KK]{kk} M.\ M.\ Kerr, A.\ Kollross: \emph{Nonnegatively curved
    homogeneous metrics in low dimensions}, Annals of Global Analysis and
    Geometry (2012) doi: 10.1007/s10455-012-9345-x


\bibitem[O]{Oniscik} A.L.\ Oni\v{s}\v{c}ik, \emph{Topology of transitive
    transformation groups}. Johann Ambrosius Barth, Leipzig (1994)


\bibitem[S]{Sch} L.~Schwachh\"ofer,  \emph{A remark on left invariant metrics
    on compact Lie groups}, Arch.\ Math.~\textbf{90} (2008) 158--162

\bibitem[ST]{ST} L.~Schwachh\"ofer, K.~Tapp, \emph{Homogeneous metrics with
    nonnegative curvature},  J.\ Geom.\ Anal.~\textbf{19}, no.~4 (2009) 929--943.

\bibitem[Wa]{Wa} N.~Wallach, \emph{Compact homogeneous Riemannian manifolds
    with strictly positive curvature}, Ann.\ of Math.\ (2) \textbf{96} (1972)
    277--295


\bibitem[Z]{Z1} W.~Ziller, \emph{Examples of Riemannian manifolds with
    non-negative sectional curvature}, Metric and Comparison Geometry, Surv.\
    Diff.\ Geom.\ \textbf{11}, ed. K.~Grove and J.~Cheeger, International Press
    (2007)

\end{thebibliography}
\end{document}